\documentclass[11pt]{amsart}
\usepackage{amssymb, amsmath, amsthm, amscd}
\newtheorem{proposition}{Proposition}[section]

\newtheorem{corollary}[proposition]{Corollary}
\newtheorem{theorem}[proposition]{Theorem}

\theoremstyle{definition}

\newcommand{\thlabel}[1]{\label{th:#1}}
\newcommand{\thref}[1]{Theorem~\ref{th:#1}}
\newcommand{\selabel}[1]{\label{se:#1}}
\newcommand{\seref}[1]{Section~\ref{se:#1}}

\newcommand{\prlabel}[1]{\label{pr:#1}}
\newcommand{\prref}[1]{Proposition~\ref{pr:#1}}

\newcommand{\eqlabel}[1]{\label{eq:#1}}
\newcommand{\equref}[1]{(\ref{eq:#1})}

\def\ot{\otimes}

\def\NN{{\mathbb N}}
\def\CC{{\mathbb C}}

\newcommand{\Cc}{\mathcal{C}}

\def\*C{{}^*\hspace*{-1pt}{\Cc}}
\def\text#1{{\rm {\rm #1}}}

\input xy
\xyoption {all} \CompileMatrices

\usepackage{amssymb}
\usepackage{color,amssymb,graphicx,amscd,amsmath}
\usepackage[colorlinks,urlcolor=blue,linkcolor=blue,citecolor=blue]{hyperref}

\begin{document}

\title[Free Poisson Hopf algebras generated by coalgebras]{Free Poisson Hopf algebras generated by coalgebras}

\author{A. L. Agore}
\address{Faculty of Engineering, Vrije Universiteit Brussel, Pleinlaan 2, B-1050 Brussels, Belgium \textbf{and} Department of Applied Mathematics, Bucharest
University of Economic Studies, Piata Romana 6, RO-010374
Bucharest 1, Romania} \email{ana.agore@vub.ac.be and
ana.agore@gmail.com}

\thanks{The author is supported by an \emph{Aspirant} Fellowship from the Fund for Scientific
Research-Flanders (Belgium) (F.W.O. Vlaanderen). This research is
part of the grant no.88/05.10.2011 of the Romanian National
Authority for Scientific Research, CNCS-UEFISCDI}



\begin{abstract}
We construct the analogue of Takeuchi's free Hopf algebra in the
setting of Poisson Hopf algebras. More precisely, we prove that
there exists a free Poisson Hopf algebra on any coalgebra or,
equivalently that the forgetful functor from the category of
Poisson Hopf algebras to the category of coalgebras has a left
adjoint. In particular, we also prove that the category of Poisson
Hopf algebras is a reflective subcategory of the category of
Poisson bialgebras. Along the way, we describe coproducts and
coequalizers in the category of Poisson Hopf algebras, therefore
showing that the latter category is cocomplete.
\end{abstract}

\maketitle

\section*{Introduction}
A Poisson Hopf algebra is both a Poisson algebra and a Hopf
algebra such that the comultiplication and the counit are
morphisms of Poisson algebras. Such objects are situated at the
border between Poisson geometry \cite{KS} and quantum groups
\cite{Dri}. Poisson structures are known for making striking and
unexpected appearances in a variety of different fields of
mathematics or mathematical physics such as differential geometry,
both classical and quantum mechanics, Lie groups and
representation theory, algebraic geometry, etc. Furthermore, as it
turns out, many important objects carry not only a Poisson
structure but also a natural Poisson Hopf algebra structure; we
only mention here the algebra of smooth functions on a Poisson
group or the polynomial algebra $\CC[\mathfrak{g}^{*}]$ of a
finite dimensional complex simple Lie algebra $\mathfrak{g}$, see
e.g. \cite{Dri, Gav1, KS, Ta} for more details and further
examples. On the other hand, there are also examples of Hopf
algebras carrying an induced Poisson structure which makes it into
a Poisson Hopf algebra. For instance, the graded algebra
associated to any connected Hopf algebra is such an example
\cite{zhu}. It is therefore important to achieve a good
understanding of these structures as well as to be able to
construct new ones. The present paper is a contribution to the
study of Poisson Hopf algebras from an algebraic point of view.
More precisely, in light of the increasing interest shown recently
in the category of Hopf algebras (see \cite{A1, A2, AC1, AC2, Kaw,
HP, HP2, HP3, RW} and the references therein), we will study the
categorical properties of Poisson Hopf algebras. As we will see,
the category $k$-PoissBiAlg of Poisson bialgebras (and therefore
the category of Poisson Hopf algebras) is not as friendly as the
category $k$-BiAlg of bialgebras in the sense that it does not
enjoy the same nice symmetry. More precisely, it is well-known
that for any field $k$ the category $k$-BiAlg of bialgebras is
isomorphic to $\mathbf{Mon}(\mathbf{Comon}({}_k{\mathcal {M}}))$
as well as to $\mathbf{Comon}(\mathbf{Mon}({}_k{\mathcal {M}}))$
while the category $k$-PoissBiAlg is only isomorphic to
$\mathbf{Comon}(k{\rm -Poiss})$, where $\mathbf{Mon}(\mathcal{C})$
and $\mathbf{Comon}(\mathcal{C})$ denote the monoids, respectively
the comonoids, of a given category $\mathcal{C}$ (see for instance
\cite{HP3}) and $k$-Poiss stands for the category of Poisson
algebras. Therefore, the main drawback in studying Poisson Hopf
algebras is that the duality arguments usually employed for Hopf
algebras do not hold anymore.

An outline of the paper is as follows. In \seref{1} we introduce
the notation and recall briefly the basic concepts needed in the
sequel. \seref{2} contains the explicit description of coproducts
and coequalizers in the categories of Poisson algebras, Poisson
bialgebras as well as Poisson Hopf algebras. As we will see, all
these constructions rely on those performed in the category of
Poisson algebras. The main results of the paper, namely the
explicit constructions of the free Poisson Hopf algebra generated
by a coalgebra, respectively a Poisson bialgebra, are proved in
\seref{3}. The constructions in this section parallel those in
\cite{T}. However, it is worth pointing out that the construction
of the free Poisson Hopf algebra on a coalgebra does not follow
trivially from Takeuchi's construction. What we mean, precisely,
is that we do not merely put a Lie algebra structure on Takeuchi's
free Hopf algebra, the free Poisson Hopf algebra on a Poisson
bialgebra being obtained by a different construction. The paper
ends with some open problems concerning the existence of limits
and cofree objects in the above mentioned categories. Another
important issue, also connected to the existence of cofree
objects, which still needs to be addressed is the injectivity
(resp. surjectivity) of monomorphisms (resp. epimorphisms) in the
category of Poisson Hopf algebras.

\section{Preliminaries}\selabel{1}

Throughout this paper, $k$ will be a field. Unless specified
otherwise, all vector spaces, tensor products, homomorphisms,
algebras, coalgebras, bialgebras, Lie algebras, Poisson algebras,
Hopf algebras and Poisson Hopf algebras are over $k$ and all
algebras are considered to be unital. Our notation for the
standard categories is as follows: ${}_k{\mathcal {M}}$
($k$-vector spaces), $k$-Alg (associative unital $k$-algebras),
$k$-Lie (Lie algebras over $k$), $k$-Poiss (Poisson algebras over
$k$), $k$-BiAlg (bialgebras over $k$), $k$-HopfAlg (Hopf algebras
over $k$), $k$-BiAlgPoiss (Poisson bialgebras over $k$),
$k$-HopfPoiss (Poisson Hopf algebras over $k$). For a coalgebra
$C$, we will use Sweedler's $\Sigma$-notation $\Delta(c)=
c_{(1)}\ot c_{(2)}$ with suppressed summation sign. $C^{\rm cop}$
stands for the coopposite of the coalgebra $C$. For a coalgebra
$C$ and an algebra $A$, ${ \rm Hom}_{k}(C, A)$ becomes an algebra
with respect to the convolution product $*$; more precisely, we
have $(f * g)(c) = f(c_{(1)})g(c_{(2)})$. If $A$ is an algebra
then the vector space $A$ together with the product $[-,\,-] : A
\times A \to A$ defined by $[a, \, b] = ab - ba$ is a Lie algebra
denoted by $A^{-}$. Given a vector space $V$, $(T(V), \, i)$
stands for the \emph{tensor algebra} on $V$, where $i : V \to
T(V)$ is the canonical inclusion. The Lie subalgebra of $T(V)^{-}$
generated by $i(V)$ is called the \emph{free Lie algebra}
generated by $V$ and will be denoted by $L(V)$. The
\emph{symmetric algebra} on $V$ will be denoted by $(S(V), \, i)$.
If $\mathfrak{g}$ is a Lie algebra with bracket $[\cdot, \,
\cdot]$ then the symmetric algebra $S(\mathfrak{g})$ inherits a
Poisson algebra structure $\{\cdot, \, \cdot\}$ in a canonical
way, namely $\{g, \, h\} = [g, \, h]$ for all $g$, $h \in
\mathfrak{g}$. Furthermore, we denote by $(\mathcal{P}(V), \,
\overline{i})$ the \emph{free Poisson algebra} on $V$, where
$\mathcal{P}(V) = S\bigl(L(V)\bigl)$ and $\overline{i}: V \to
\mathcal{P}(V)$ is the canonical inclusion; see \cite{freePo} and
the references therein for more details on the free Poisson
algebra. As the terminology suggests, the functor sending a vector
space $V$ to the free Poisson algebra $\mathcal{P}(V)$ provides a
left adjoint to the forgetful functor $U: k$-Poiss $\to
{}_k{\mathcal {M}}$.

Let us recall briefly some well known results pertaining to
category theory, referring the reader to \cite{AR, ML} for more
details. A category $\mathcal{C}$ is called \textit{(co)complete}
if all diagrams in $\mathcal{C}$ have (co)limits in $\mathcal{C}$.
A category $\mathcal{C}$ is \textit{(co)complete} if and only if
$\mathcal{C}$ has (co)equalizers of all pairs of arrows and all
(co)products \cite[Theorem 6.10]{par}. For a more detailed
discussion concerning the completeness and cocompleteness of some
of the categories mentioned above we refer the reader to \cite{A1,
A2, HP, HP2}. A full subcategory $\mathcal{D}$ of $\mathcal{C}$ is
called \textit{(co)reflective} in $\mathcal{C}$ when the inclusion
functor $U: \mathcal{D} \rightarrow \mathcal{C}$ has a (right)
left adjoint. A very convenient way of proving, in a constructive
way, that a given covariant functor $F: \mathcal{C} \to
\mathcal{D}$ has a left adjoint is by showing that for any object
$X \in \mathcal{D}$ the co-universal problem generated by $X$ and
$F$ has a co-universal solution. More precisely, given $X \in
\mathcal{D}$, a co-universal solution for the co-universal problem
generated by $X$ and $F$ consists of an object $G(X) \in
\mathcal{C}$ and a map $i: X \to F(G(X))$ in $\mathcal{D}$ such
that for each $Y \in \mathcal{C}$ and for each map $f: X \to F(Y)$
in $\mathcal{D}$ there is a unique map $g: G(X) \to Y$ in
$\mathcal{C}$ such that the following diagram commutes:
$$
\xymatrix  {& X \ar[r]^-{i} \ar[dr]_-{f} & {F(G(X))} \ar[d]^-{F(g)}\\
& {} & F(Y) }
$$
If for each $X \in \mathcal{D}$ the co-universal problem defined
by $X$ and $F$ has a co-universal solution, then $G: \mathcal{D}
\to \mathcal{C}$ defines a functor which is a left adjoint to $F$
\cite{P}.

Recall that a \emph{Poisson algebra} is both an associative
commutative algebra and a Lie algebra living on the same vector
space $P$ such that any hamiltonian $[p, -] : P \to P$ is a
derivation of the associative algebra P, i.e. for all $p$, $q$, $r
\in P$ we have:
$$[p,\, qr] = [p,\,q]\,r + q\, [p,\, r]$$

If $P_{1}$, $P_{2}$ are Poisson algebras then $P_{1} \otimes
P_{2}$ has a Poisson algebra structure defined for all $p$, $r\in
P_1$ and $q$, $s\in P_2$ by:
\begin{eqnarray}
(p \otimes q) \cdot (r \otimes s) := pr \otimes qs, \qquad \left[p
\otimes q, \, r \otimes s\right]_{P \ot P} := pr \otimes [q, \, s]
+ [p, \, r] \otimes qs \eqlabel{1.2}
\end{eqnarray}
A linear map $f: P_{1} \to P_{2}$ is called a \emph{morphism of
Poisson algebras} if $f$ is both an algebra map as well as a Lie
algebra map. Furthermore, the category $k$-Poiss of Poisson
algebras is in fact a monoidal category with the tensor product
defined above. A \emph{Poisson ideal} is a linear subspace which
is both an ideal with respect to the associative product as well
as a Lie ideal. If $\mathcal{I}$ is a Poisson ideal of $P$ then
$P/\mathcal{I}$ inherits a Poisson algebra structure in the
obvious way.

A commutative bialgebra $B$ together with a Poisson bracket
$[\cdot, \, \cdot]_{B}$ is called a \emph{Poisson bialgebra} if
the comultiplication $\Delta_{B}$ and the counit $\varepsilon_{B}$
are Poisson algebra maps, i.e. besides from being algebra maps,
for all $a$, $b \in B$ we also have:
\begin{equation}\eqlabel{defPH}
\Delta_{B}\bigl([a, \, b]\bigl) = [\Delta_{B}(a), \,
\Delta_{B}(b)]_{B \ot B}, \qquad \varepsilon_{B}\bigl([a, \,
b]\bigl) = [\varepsilon_{B}(a), \, \varepsilon_{B}(b)]_{k}
\end{equation}
Let us observe that the second compatibility in \equref{defPH} is
trivially fulfilled as for all $a$, $b \in B$ we have
$\varepsilon_{B}\bigl([a, \, b]\bigl) = 0$ (see \cite{LWZ}).
Furthermore, if $B$ is a Hopf algebra then $B$ is called a
\emph{Poisson Hopf algebra}. It is straightforward to see that the
antipode $S_{B}$ is a Poisson algebra anti-morphism, i.e. for all
$a$, $b \in B$ we have: $S_{B}\bigl([a, \, b]_{B}\bigl) =
[S_{B}(b), \, S_{B}(a)]_{B}$. A \emph{morphism of Poisson
bialgebras} is both a morphism of Poisson algebras and a morphism
of coalgebras. A morphism of Poisson bialgebras between two
Poisson Hopf algebras is automatically a Poisson Hopf morphism and
therefore $k$-HopfPoiss is a full subcategory of the category
$k$-BiAlgPoiss. It is straightforward to see that if $(H, m, \eta,
\Delta, \varepsilon, S, [-,\,-])$ is a Poisson Hopf algebra then
$(H^{\rm cop}, m, \eta, \Delta, \varepsilon, S, [-,\,-]^{\rm
op})$, where $[a, \, b]^{\rm op} = [b,\, a]$, is again a Poisson
Hopf algebra which we will denote by $H^{\rm op, cop}$. We refer
to \cite{Sw} for further details concerning Hopf algebras and to
\cite{LPV} for a comprehensive treatment of Poisson algebras from
both algebraic and geometrical point of view.

\section{Colimits in the category of Poisson Hopf algebras}\selabel{2}
The aim of this section is to construct coproducts and
coequalizers in the category $k$-PoissHopf of Poisson Hopf
algebras. Since there is a close connection between these
constructions and the ones corresponding to the category $k$-Poiss
of Poisson algebras we start by investigating the latter category.
Our next result constructs coproducts and coequalizers in the
category of Poisson algebras; as we will see, these constructions
are obtained by modifying properly the objects which provide the
coproducts respectively the coequalizers in the category $k$-Alg
of algebras \cite[page 50]{P}.

\begin{proposition}\prlabel{0.0.1}
The category $k$-Poiss of Poisson algebras has arbitrary
coproducts and coequalizers. In particular, the category $k$-Poiss
of Poisson algebras is cocomplete.
\end{proposition}

\begin{proof}
We first indicate the construction of coproducts. Let $(P_{l})_{l
\in I}$ be a family of Poisson algebras and consider
$\bigl(\bigoplus_{l \in I}P_{l}, \, (j_{l})_{l \in I}\bigl)$ to be
their coproduct in ${}_k{\mathcal {M}}$. Then $\bigl(\coprod_{l
\in I} P_{l} := S\bigl(\bigoplus_{l \in I}P_{l}\bigl)/\overline{J}
, \, \{-,\, -\}, \, (q_{l})_{l \in I}\bigl)$ is the coproduct of
the above family in $k$-Poiss, where $(S \bigl(\bigoplus_{l \in
I}P_{l}\bigl), \, i)$ is the symmetric algebra on the vector space
$\bigoplus_{l \in I}P_{l}$, $i: \bigoplus_{l \in I}P_{l}
\rightarrow S\bigl(\bigoplus_{l \in I}P_{l}\bigl)$ stands for the
canonical inclusion, $\overline{J}$ is the Poisson ideal generated
by the set $J:= \{i\circ
j_{l}(x_{l}y_{l})-i\bigl(j_{l}(x_{l})\bigl)i\bigl(j_{l}(y_{l})\bigl)
,$ $1_{S\bigl(\bigoplus P_{l}\bigl)} - i\circ j_{l}(1_{P_{l}}) \,
| \, x_{l}, y_{l} \in P_{l}, l \in I \}$, $\nu :
S\bigl(\bigoplus_{l \in I}P_{l}\bigl) \rightarrow
S\bigl(\bigoplus_{l \in I}P_{l}\bigl)/\overline{J}$ denotes the
canonical projection and $q_{l} = \nu \circ i \circ j_{l}$, for
all $l \in I$. Since the idea behind this construction is
essentially the same as the one used in the case of associative
algebras we will be brief. Consider $Q$ to be a Poisson algebra
and $u_{r}: P_{r} \to Q$ a family of Poisson algebra maps.
Composing the universal maps depicted in \equref{PoissCo} it
yields a unique Poisson algebra map $u: S\bigl(\bigoplus_{l \in
I}P_{l}\bigl)/\overline{J} \to Q$ such that $u \circ q_{r} =
u_{r}$:
\newpage
\begin{equation}\eqlabel{PoissCo}
\xymatrix  @R=4pc {P_{r}\ar[rr]^-{j_{r}}\ar[drrr]_{u_{r}} & {} &
{\bigoplus_{l \in I}P_{l}}\ar[rr]^{i}\ar[dr]_{\overline{u}} & {} &
{S\bigl(\bigoplus
P_{l}\bigl)}\ar[rr]^-{\nu}\ar[dl]_{\underline{u}} & {} &
{S\bigl(\bigoplus_{l \in I}P_{l}\bigl)/\overline{J}}\ar[dlll]^{u}\\
{} & {} & {} & Q & {} & {} & {}}
\end{equation}
Next in line are coequalizers. Let $f$, $g: P \to Q$ be two
Poisson algebra maps and consider $\overline{I}$ to be the Poisson
ideal generated by the set $\{f(p)-g(p) \, | \, p \in P \, \}$.
Then $(Q/\overline{I}, \pi)$ is the coequalizer of the morphisms
$(f,g)$ in $k$-Poiss, where $\pi:Q \rightarrow Q/\overline{I}$ is
the canonical projection; the details are straightforward and left
to the reader.
\end{proof}

In order to have a complete picture on the category $k$-Poiss of
Poisson algebras we record here the result concerning limits:

\begin{proposition}\prlabel{PoissL}
The category $k$-Poiss of Poisson algebras has arbitrary products
and equalizers. In particular, the category $k$-Poiss of Poisson
algebras is complete.
\end{proposition}
\begin{proof}
It is straightforward to see that both products as well as
equalizers can be constructed as simply the products and
respectively the equalizers of the underlying vector spaces with
the obvious commutative algebra and Lie algebra structures.
\end{proof}

As mentioned before, the colimits construction in the categories
$k$-BiAlgPoiss and $k$-HopfPoiss rely heavily on the corresponding
construction performed in the category $k$-Poiss.

\begin{theorem}\thlabel{1.3}
The categories $k$-BiAlgPoiss and $k$-HopfPoiss of Poisson
bialgebras and respectively Poisson Hopf algebras have arbitrary
coproducts and coequalizers. In particular, the above categories
are cocomplete.
\end{theorem}
\begin{proof}
First we deal with products in the category of Poisson bialgebras.
Consider a family of Poisson bialgebras $\bigl(B_{i}, m_{i},
\eta_{i}, \Delta_{i}, \varepsilon_{i}, [-,\, -]_{i}\bigl)_{i \in
I}$ and let $\bigl(\coprod_{i \in I} B_{i},\, (q_{i})_{i \in
I}\bigl)$ be the coproduct of this family in the category
$k$-Poiss as described in \prref{0.0.1}. It turns out that
$\coprod_{i \in I} B_{i}$ is actually a Poisson bialgebra with
comultiplication and counit given by the unique Poisson algebra
maps such that the following diagrams commute:
\begin{equation}\eqlabel{B}
\xymatrix {& B_{l} \ar[r]^-{q_{l}} \ar[dr]_-{(q_{l} \otimes q_{l})\circ \Delta_{l}} & {\coprod_{i \in I} B_{i}} \ar[d]^{\Delta}\\
& {} & {\coprod_{i \in I} B_{i} \otimes \coprod_{i \in I} B_{i}}}
\xymatrix {& B_{l} \ar[r]^-{q_{l}} \ar[dr]_{\varepsilon_{l}} & {\coprod_{i \in I} B_{i}} \ar[d]^{\varepsilon}\\
& {} & k }
\end{equation}
Proving that $(\coprod_{i \in I} B_{i}, \, \Delta, \,
\varepsilon)$ is a coalgebra goes essentially in the same vain as
in the case of Hopf algebras; we refer the reader to \cite{P} for
more details.

Next we look at coequalizers. Let $f$, $g: P \to Q$ be two Poisson
bialgebra maps and consider $\overline{I}$ to be the Poisson ideal
of $Q$ generated by the set $\{f(p)-g(p) \, | \, p \in P \, \}$.
We prove that $\overline{I}$ is also a coideal, i.e.
$\Delta(\overline{I}) \subseteq Q \ot \overline{I} + \overline{I}
\ot Q$. In fact, as $\Delta$ is a Poisson algebra morphism we only
need to check that $\Delta\bigl([q, \, r(f(p) - g(p))]\bigl)
\subseteq Q \ot \overline{I} + \overline{I} \ot Q$, for all $q$,
$r \in Q$ and $p \in P$. Indeed, using induction this would imply
that $\Delta\bigl([q_{1}, \, [q_{2}, \, ..., [q_{k}, \, r(f(p) -
g(p))]]]\bigl) \subseteq Q \ot \overline{I} + \overline{I} \ot Q$
for all $k \in \NN$, $q_{1}$, $q_{2}$, ..., $q_{k}$, $r \in Q$, $p
\in P$ and the conclusion follows. Now we prove our original
claim:
\begin{eqnarray*}
&&\Delta\bigl([q, \, r(f(p) - g(p))]\bigl) = \bigl[\Delta(q), \,
\Delta \bigl(r(f(p) - g(p))\bigl)\bigl]_{Q \ot Q}\\
&=& \bigl[q_{(1)} \ot q_{(2)}, \, (r_{(1)} \ot r_{(2)})
\bigl(f(p_{(1)}) \ot f(p_{(2)}) - g(p_{(1)}) \ot g(p_{(2)})\bigl)\bigl]_{Q \ot Q}\\
&=& \bigl[q_{(1)} \ot q_{(2)}, \, (r_{(1)} \ot r_{(2)})
\bigl(f(p_{(1)}) \ot (f(p_{(2)}) - g(p_{(2)})) \bigl) -\\
&& (r_{(1)} \ot r_{(2)})\bigl((f(p_{(1)}) - g(p_{(1)})) \ot g(p_{(2)})\bigl)\bigl]_{Q \ot Q}\\
&=& \bigl[q_{(1)} \ot q_{(2)}, \, r_{(1)} f(p_{(1)}) \ot r_{(2)}
(f(p_{(2)}) - g(p_{(2)})) \bigl]_{Q \ot Q} - \\
&& \bigl[q_{(1)} \ot q_{(2)}, \, r_{(1)}(f(p_{(1)}) - g(p_{(1)})) \ot r_{(2)} g(p_{(2)})\bigl]_{Q \ot Q}\\
&\stackrel{\equref{1.2}} {=}& q_{(1)}  r_{(1)} f(p_{(1)}) \ot
\underline{[q_{(2)}, \, r_{(2)}
(f(p_{(2)}) - g(p_{(2)}))]} + \\
&& [q_{(1)}, \,  r_{(1)} f(p_{(1)})] \ot \underline{q_{(2)} r_{(2)}
(f(p_{(2)}) - g(p_{(2)}))} -\\
&& \underline{q_{(1)} r_{(1)}(f(p_{(1)}) - g(p_{(1)}))} \ot
[q_{(2)}, \,r_{(2)} g(p_{(2)})] - \\
&& \underline{[q_{(1)},\, r_{(1)}(f(p_{(1)}) - g(p_{(1)}))]} \ot
q_{(2)} r_{(2)} g(p_{(2)})
\end{eqnarray*}
and the last line is obviously in $Q \ot \overline{I} +
\overline{I} \ot Q$, as the underlined terms are in
$\overline{I}$. Therefore, $Q/\overline{I}$ is a Poisson bialgebra
in a canonical way and moreover $(Q/\overline{I}, \pi)$ is the
coequalizer of the morphisms $(f,g)$ in $k$-BiAlgPoiss, where
$\pi:Q \rightarrow Q/\overline{I}$ is the canonical projection.

Consider now a family of Poisson Hopf algebras $\bigl(H_{i},
m_{i}, \eta_{i}, \Delta_{i}, \varepsilon_{i}, S_{i},
[-,\,-]_{i}\bigl)_{i \in I}$ and let $\bigl(H:=\coprod_{i \in I}
H_{i},\, m, \, \eta, \, \Delta, \, \varepsilon,\, [-,\, -], \,
(q_{i})_{i \in I}\bigl)$ be the previously constructed coproduct
of the underlying Poisson bialgebras. Remark that $S_{i}: H_{i}
\to {H_{i}}^{\rm op, cop}$ is a Poisson bialgebra map. Then, the
universal property of the coproduct yields an unique Poisson
bialgebra map $S:H \rightarrow H^{\rm op, cop}$ such that the
following diagram commutes for all $i \in I$:
\begin{equation}\eqlabel{antipod}
\xymatrix {& H_{i} \ar[r]^{q_{i}} \ar[dr]_{q_{i} \circ S_{i}} & {H} \ar[d]^{S}\\
& {} & {H^{\rm op, cop}}}
\end{equation}
As $S:H \rightarrow H^{\rm op, cop}$ defined in \equref{antipod}
is a Poisson bialgebra map we only need to prove that $S$ is
indeed an antipode for $H$. This follows exactly as in the proof
of \cite[Theorem 2.2]{A2}. Finally, since $k$-HopfPoiss is a full
subcategory of the category $k$-BiAlgPoiss it follows that
$\bigl(H:=\coprod_{i \in I} H_{i},\, \Delta, \, \varepsilon,\,
S,\, [-,\, -], \, (q_{i})_{i \in I}\bigl)$ is also
the coproduct in $k$-HopfPoiss. \\
Consider now $f$, $g: P \to Q$ to be two Poisson Hopf algebra maps
and, as before, let $\overline{J}$ be the Poisson ideal of $Q$
generated by the set $\{f(p)-g(p) \, | \, p \in P \, \}$. The
computations performed in the case of Poisson bialgebras imply
that $\overline{j}$ is also a coideal. We are left to prove that
$\overline{j}$ is a Hopf ideal, i.e. $S_{Q}(\overline{J})
\subseteq \overline{J}$. Arguing as in the first part of the
proof, we only need to show that $S_{Q} \bigl([q,\, r(f(p) -
g(p))\bigl) \subseteq \overline{J}$ for all $q$, $r \in Q$ and $p
\in P$. Indeed, we have:
\begin{eqnarray*}
S_{Q} \bigl([q,\, r(f(p) - g(p))\bigl) &=& \bigl[S_{Q}\bigl(r(f(p)
- g(p))\bigl), \, S_{Q}(q)\bigl]\\
&=& \bigl[S_{Q}\bigl(f(p)
- g(p)\bigl)\, S_{Q}(r), \, S_{Q}(q)\bigl]\\
&=& \bigl[\bigl(f(S_{P}(p)) - g(S_{P}(p))\bigl) S_{Q}(r), \,
S_{Q}(q)\bigl]\\
&=& - \bigl[S_{Q}(q),\, S_{Q}(r)\bigl(f(S_{P}(p)) -
g(S_{P}(p))\bigl) \bigl] \in \overline{J}
\end{eqnarray*}
and the proof is now finished.
\end{proof}

\section{The free Poisson Hopf algebra on a coalgebra}\selabel{3}

In this section we introduce the main characters of this paper,
namely the free Poisson Hopf algebras generated by coalgebras. The
strategy we pursue is the following: first, we introduce the free
Poisson bialgebra on a coalgebra (\thref{free1}) and then we prove
that there also exist a free Poisson Hopf algebra on every Poisson
bialgebra (\thref{free2}). Finally, by putting the two
constructions together we arrive at the free Poisson Hopf algebra
generated by a coalgebra (\thref{free3}).

\begin{theorem}\thlabel{free1}
The forgetful functor $F_{1}: k$-$BiAlgPoiss \to k$-$CoAlg$ has a
left adjoint, i.e. there exists a free Poisson bialgebra on every
coalgebra.
\end{theorem}
\begin{proof}
Let $(C,\, \Delta_{C}, \, \varepsilon_{C})$ be a coalgebra and
consider $(\mathcal{P}(C),\,\overline{i})$ to be the free Poisson
algebra on the vector space $C$. By the universal property of the
free Poisson algebra we obtain two Poisson algebra maps
$\overline{\Delta}: \mathcal{P}(C) \to \mathcal{P}(C) \ot
\mathcal{P}(C) $ and respectively $\overline{\varepsilon}:
\mathcal{P}(C) \to k$ such that the following two diagrams
commute:
\begin{equation}\eqlabel{bialg}
\xymatrix @R=4pc  {& C \ar[r]^-{\overline{i}} \ar[dr]_-{(\overline{i} \otimes \overline{i})\circ \Delta_{C}} & {\mathcal{P}(C)} \ar[d]^-{\overline{\Delta}}\\
& {} & {\mathcal{P}(C) \otimes \mathcal{P}(C)}}
\xymatrix {& C \ar[r]^-{\overline{i}} \ar[dr]_-{\varepsilon_{C}} & {\mathcal{P}(C)} \ar[d]^-{\overline{\varepsilon}}\\
& {} & k }
\end{equation}
We start by proving that $(\mathcal{P}(C),\,
\overline{\Delta},\,\overline{\varepsilon})$ is a coalgebra.
Consider the Poisson algebra map $(\overline{\Delta} \ot Id) \circ
\overline{\Delta} \circ \overline{i}$. The universal property of
the free Poisson algebra yields a unique Poisson algebra map
$\psi: \mathcal{P}(C) \to \mathcal{P}(C) \otimes \mathcal{P}(C)
\otimes \mathcal{P}(C)$ such that the following diagram commutes:
\begin{equation}
\xymatrix @R=4pc  {& C \ar[r]^-{\overline{i}} \ar[dr]_-{(\overline{\Delta} \ot Id) \circ \overline{\Delta} \circ\overline{i}}  & {\mathcal{P}(C)} \ar[d]^-{\psi}\\
& {} & {\mathcal{P}(C) \otimes \mathcal{P}(C) \otimes
\mathcal{P}(C)}}
\end{equation}
It is easy to see that the Poisson algebra map $(\overline{\Delta}
\ot Id) \circ \overline{\Delta}$ makes the above diagram commute.
Thus, using the uniqueness of $\psi$, in order to prove that
$(\overline{\Delta} \ot Id) \circ \overline{\Delta} = (Id \ot
\overline{\Delta}) \circ \overline{\Delta}$ it is enough to show
that $(\overline{\Delta} \ot Id) \circ \overline{\Delta} \circ
\overline{i} = (Id \ot \overline{\Delta}) \circ \overline{\Delta}
\circ \overline{i}$. Indeed, concerning the last claim we have:
\begin{eqnarray*}
(\overline{\Delta} \ot Id) \circ \underline{\overline{\Delta}
\circ \overline{i}} &\stackrel{\equref{bialg}} {=}&
(\overline{\Delta} \ot Id) \circ (\overline{i} \ot \overline{i})
\circ \Delta_{C}\\
&{=}& \bigl((\underline{\overline{\Delta} \circ \overline{i}}) \ot
\overline{i}\bigl) \circ \Delta_{C}\\
&\stackrel{\equref{bialg}} {=}& \Bigl(\bigl((\overline{i} \ot
\overline{i}) \circ \Delta_{C}\bigl) \ot  \overline{i}\Bigl) \circ
\Delta_{C}\\
&{=}& (\overline{i} \ot \overline{i} \ot \overline{i}) \circ
\underline{(Id \ot \Delta_{C}) \circ \Delta_{C}}\\
&{=}& (\overline{i} \ot \overline{i} \ot \overline{i}) \circ
(\Delta_{C} \ot Id) \circ \Delta_{C}\\
&{=}& \Bigl(\bigl(\overline{i} \ot \underline{(\overline{i} \ot \overline{i})
\circ \Delta_{C}} \bigl) \Bigl) \circ \Delta_{C}\\
&\stackrel{\equref{bialg}} {=}& \bigl(\overline{i} \ot
(\overline{\Delta} \circ \overline{i})\bigl) \circ \Delta_{C}\\
&{=}& (Id \ot \overline{\Delta}) \circ \underline{(\overline{i}
\ot \overline{i}) \circ \Delta_{C}}\\
&\stackrel{\equref{bialg}} {=}& (Id \ot \overline{\Delta}) \circ
\overline{\Delta} \circ \overline{i}
\end{eqnarray*}
A similar argument proves that $(Id \ot \overline{\varepsilon})
\circ \overline{\Delta} = (\overline{\varepsilon} \ot Id) \circ
\overline{\Delta} = Id$ and thus $\mathcal{P}(C)$ is in fact a
Poisson bialgebra. The proof will be finished once we show that
the pair $(\mathcal{P}(C), \overline{i})$ provides a co-universal
solution to the co-universal problem generated by the coalgebra
$C$ and the forgetful functor $F_{1}: k$-BiAlgPoiss $\to k$-CoAlg.
To this end, let $H$ be a Poisson bialgebra and $f: C \to H$ a
coalgebra map. By the universal property of the free Poisson
algebra, there exists a Poisson algebra map $\overline{f} :
\mathcal{P}(C) \to H$ such that the following diagram is
commutative:
$$
\xymatrix {& C \ar[r]^-{\overline{i}} \ar[dr]_-{f} & {\mathcal{P}(C)} \ar[d]^-{\overline{f}}\\
& {} & H }
$$
We are left to prove that $\overline{f}$ is also a coalgebra map.
Since $\Delta_{H} \circ \overline{f}$ is a Poisson algebra map, by
the universal property of the free Poisson algebra, there exists a
unique Poisson algebra map $\xi: \mathcal{P}(C) \to H \ot H$ such
that the following diagram commutates:
$$
\xymatrix @R=4pc  {& C \ar[r]^-{\overline{i}} \ar[dr]_-{\Delta_{H} \circ \overline{f} \circ \overline{i}} & {\mathcal{P}(C)} \ar[d]^-{\xi}\\
& {} & {H \otimes H}}
$$
Obviously, $\Delta_{H} \circ \overline{f}$ makes the above diagram
commutative. By the same argument as before, in order to prove
that $(\overline{f} \ot \overline{f}) \circ \overline{\Delta} =
\Delta_{H} \circ \overline{f}$, it will be enough to show that
$(\overline{f} \ot \overline{f}) \circ \overline{\Delta} \circ
\overline{i} = \Delta_{H} \circ \overline{f} \circ \overline{i}$.
As for the last claim, we have:
\begin{eqnarray*}
(\overline{f} \ot \overline{f}) \circ \underline{\overline{\Delta}
\circ \overline{i}} &\stackrel{\equref{bialg}} {=}& (\overline{f}
\ot \overline{f}) \circ (\overline{i} \ot \overline{i}) \circ
\Delta_{C}\\
&{=}& \bigl((\underline{\overline{f} \circ \overline{i}}) \ot
(\underline{\overline{f}
\circ \overline{i}})\bigl) \circ \Delta_{C}\\
&{=}& (f \ot f) \circ \Delta_{C}\\
&\stackrel{f {\rm~coalgebra~map~}}{=}& \Delta_{H} \circ f =
\Delta_{H} \circ \overline{f} \circ \overline{i}
\end{eqnarray*}
Consider now the Poisson algebra map $\varepsilon_{H} \circ
\overline{f} \circ \overline{i}$. Again by the universal property
of the free Poisson algebra, there exists a unique Poisson algebra
map $\chi: \mathcal{P}(C) \to k$ such that the following diagram
is commutative:
$$
\xymatrix  {& C \ar[r]^-{\overline{i}} \ar[dr]_-{\varepsilon_{H} \circ \overline{f} \circ \overline{i}} & {\mathcal{P}(C)} \ar[d]^-{\chi}\\
& {} & {k}}
$$
Using the same argument as before, in order to prove that
$\varepsilon_{H} \circ \overline{f} = \overline{\varepsilon}$ we
only need to show that $\varepsilon_{H} \circ \overline{f} \circ
\overline{i} = \overline{\varepsilon} \circ \overline{i}$. The
latter statement holds true by the following straightforward
computation:
\begin{eqnarray*}
\varepsilon_{H} \circ \overline{f} \circ \overline{i} =
\varepsilon_{H} \circ f = \varepsilon_{C} = \overline{\varepsilon}
\circ \overline{i}
\end{eqnarray*}
This finishes the proof.
\end{proof}

In the proof of our next theorem we make use of some statements
already proven in \cite[Theorem 2.6.3]{P}. Since the proof there
is very detailed and freely available, we will refer to it and
leave out those computations.

\begin{theorem}\thlabel{free2}
The forgetful functor $F_{2}: k$-$PoissHopf \to k$-$BiAlgPoiss$
has a left adjoint, i.e. there exists a free Poisson Hopf algebra
on every Poisson bialgebra.
\end{theorem}
\begin{proof}
Let $B$ be a Poisson bialgebra. We aim to construct a co-universal
solution to the co-universal problem generated by $B$ and the
forgetful functor $F_{2}: k$-PoissHopf $\to k$-BiAlgPoiss. To this
end, consider $\bigl(C = \coprod_{n \in \NN} B_{i},\, (q_{n})_{n
\in \NN}\bigl)$ to be the coproduct in $k$-BiAlgPoiss of the
Poisson bialgebras $B_{n}$, $n \geq 0$, where $B_{n} = B$ for $n$
even and $B_{n} = B^{\rm op, cop}$ for $n$ odd. The universality
of the coproduct yields a unique Poisson bialgebra map $S': C \to
C^{\rm op, cop}$ which makes the following diagram commutative:
$$
\xymatrix {& B_{i} \ar[rr]^{q_{i}} \ar[dr]_{Id} & {} & {C} \ar[dd]^{S'}\\
& {} & B_{i+1}^{\rm op, cop}\ar[dr]_{q_{i+1}} & {} \\
& {} & {} & {C^{\rm op, cop}}}
$$
Now consider $\mathcal{I}$ to be the Poisson ideal of $C$
generated by the set $I = \{(S' * Id - u_{C} \circ
\varepsilon_{C})(q_{n}(x)), \, (Id * S'- u_{C} \circ
\varepsilon_{C})(q_{n}(x)) ~|~ x \in B_{n},\, n \in \NN \}$. We
will prove that $\mathcal{I}$ is also a Hopf ideal, i.e.
$\varepsilon_{C}(\mathcal{I}) = 0$, $\Delta_{C}(\mathcal{I})
\subseteq C \ot \mathcal{I} + \mathcal{I} \ot C$ and
$S'(\mathcal{I}) \subseteq \mathcal{I}$. As noted in the proof of
\thref{1.3}, it is enough to check this on the elements of the
form $\bigl[a, \, b\bigl(S'(q_{n}(x)_{(1)}) q_{n}(x)_{(2)} -
u_{C}\circ \varepsilon_{C} (q_{n}(x)) \bigl)\bigl]$ where $a$, $b
\in C$ and $x \in B_{n}$, $n \in \NN$. We know from \cite[Theorem
2.6.3]{P} that $\Delta_{C}(S'(q_{n}(x)_{(1)}) q_{n}(x)_{(2)} -
u_{C}\circ \varepsilon_{C} (q_{n}(x))) \in C \ot CI + CI \ot C$,
where $CI$ denotes the two-sided ideal generated by $I$. Since the
inclusion $CI \subset \mathcal{I}$ holds true, we will denote
$\Delta_{C}(S'(q_{n}(x)_{(1)}) q_{n}(x)_{(2)} - u_{C}\circ
\varepsilon_{C} (q_{n}(x)))  = c \ot \iota + \overline{\iota} \ot
\overline{c}$ (note that in order to be consistent with our
notations, we suppressed the summation sign in the right hand
side) with $c$, $\overline{c} \in C$ and $\iota$,
$\overline{\iota} \in \mathcal{I}$. Then, we have:
\begin{eqnarray*}
&&\Delta_{C} \Bigl(\bigl[a, \, b\bigl(S'(q_{n}(x)_{(1)})
q_{n}(x)_{(2)} - u\circ \varepsilon (q_{n}(x))
\bigl)\bigl]_{C}\Bigl) =\\
&=& \bigl[\Delta(a),\, \Delta(b) \Delta \bigl(S'(q_{n}(x)_{(1)})
q_{n}(x)_{(2)} - u\circ
\varepsilon (q_{n}(x)) \bigl) \bigl]_{C \ot C}\\
&=& \bigl[a_{(1)} \ot a_{(2)}, \, (b_{(1)} \ot b_{(2)}) (c \ot
\iota + \overline{\iota} \ot \overline{c})\bigl]_{C \ot C}\\
&\stackrel{\equref{1.2}} {=}& \bigl[a_{(1)} \ot a_{(2)}, \, b_{(1)} c \ot b_{(2)} \iota \bigl]_{C \ot C} + \bigl[a_{(1)} \ot a_{(2)}, \, b_{(1)} \overline{\iota} \ot b_{(2)} \overline{c} \bigl]_{C \ot C}\\
&\stackrel{\equref{1.2}} {=}& a_{(1)} b_{(1)} c \ot
\underline{[a_{(2)}, \, b_{(2)} \iota]_{C}} + [a_{(1)}, \, b_{(1)}
c]_{C} \ot \underline{a_{(2)}  b_{(2)} \iota} + \underline{a_{(1)}
b_{(1)} \overline{\iota}} \ot [a_{(2)}, \, b_{(2)}
\overline{c}]_{C} \\
&& + \underline{[a_{(1)}, \, b_{(1)} \overline{\iota}]_{C}} \ot
a_{(2)} b_{(2)} \overline{c}
\end{eqnarray*}
and the conclusion follows since the underlined terms belong to
$\mathcal{I}$. We are left to show that $S'(\mathcal{I}) \subseteq
\mathcal{I}$. Recall from the proof of \cite[Theorem 2.6.3]{P}
that $S'\bigl(S'(q_{n}(x)_{(1)}) q_{n}(x)_{(2)} - u_{C}\circ
\varepsilon_{C} (q_{n}(x)) \bigl) = q_{n+1}(x)_{(1)} S'(
q_{n+1}(x)_{(2)}) - u_{C}\circ \varepsilon_{C} (q_{n+1}(x)) \in I$
for all $x \in B_{n}$, $n \in \NN$. Therefore, we have:
\begin{eqnarray*}
&& S' \Bigl(\bigl[a, \, b\bigl(S'(q_{n}(x)_{(1)}) q_{n}(x)_{(2)} -
u\circ \varepsilon (q_{n}(x)) \bigl)\bigl]_{C}\Bigl) =\\
&=& \Bigl[S'\bigl(S'(q_{n}(x)_{(1)}) q_{n}(x)_{(2)} - u\circ
\varepsilon (q_{n}(x)) \bigl) S'(b), \, S'(a)\Bigl]_{C}\\
&=& \Bigl[ \underline{\bigl(q_{n+1}(x)_{(1)} S'( q_{n+1}(x)_{(2)}) - u\circ
\varepsilon (q_{n+1}(x))\bigl)} S'q_{n+1}(x)_{(1)} S'(b), \, S'(a)\Bigl]_{C}\\
\end{eqnarray*}
and the last line is clearly in $\mathcal{I}$ since the underlined
term belongs to $I$.

We are now ready to construct a co-universal solution to the
co-universal problem generated by the Poisson bialgebra $B$ and
the functor $F_{2}: k$-PoissHopf $\to k$-BiAlgPoiss. To this end,
consider the Poisson algebra $H(C): = C/ \mathcal{I}$ and $\pi: C
\to C/\mathcal{I}$ the canonical projection. $H(C)$ can be made
into a Poisson Hopf algebra with the coalgebra structure
$(\overline{\Delta},\, \overline{\varepsilon})$ and antipode
$\overline{S}$ given by the unique Poisson algebra maps which make
the following diagrams commute:
$$
\xymatrix  {& C \ar[r]^-{\pi} \ar[dr]_-{\varepsilon } & {H(C)} \ar[d]^-{\overline{\varepsilon}}\\
& {} & {k}}
\xymatrix @R=4pc {& C \ar[r]^-{\pi} \ar[dr]_{(\pi \ot \pi) \circ \Delta} & {H(C)} \ar[d]^{\overline{\Delta}}\\
& {} & {H(C) \ot H(C)}}
\xymatrix @R=4pc {& C \ar[r]^-{\pi} \ar[dr]_{\pi \circ S'} & {H(C)} \ar[d]^{\overline{S}}\\
& {} & {H(C)^{\rm cop}}}
$$
By arguments similar to those used in the proof of \cite[Theorem
2.6.3]{P} it can be easily seen that the object constructed above
indeed provides a co-universal solution to the co-universal
problem generated by the Poisson bialgebra $B$ and the functor
$F_{2}: k$-PoissHopf $\to k$-BiAlgPoiss.
\end{proof}

We can now put the last two results together:

\begin{theorem}\thlabel{free3}
The forgetful functor $F: k$-PoissHopf $\to k$-CoAlg has a left
adjoint, i.e. there exists a free Poisson Hopf algebra on every
coalgebra.
\end{theorem}
\begin{proof}
It follows by composing the two left adjoint functors constructed
in \thref{free1} and \thref{free2}.
\end{proof}

\begin{corollary}
The categories $k$-PoissHopf and $k$-BiAlgPoiss of Poisson Hopf
algebras and respectively Poisson bialgebras have generators.
\end{corollary}
\begin{proof}
We denote by $\overline{F}_{1}:$ k-CoAlg $\to k$-BiAlgPoiss and
$\overline{F}:$ k-CoAlg $\to k$-PoissHopf the left adjoint
functors of $F_{1}: $k-BiAlgPoiss $\to k$-CoAlg and respectively
$F: k$-PoissHopf $\to k$-CoAlg. By \cite[Theorem 17]{PS}, the
category $k$-CoAlg of coalgebras has a generator $G$. Recall that
$G = \coprod_{x \in I} C_{x}$, where $I$ is the set of isomorphism
classes of finite dimensional coalgebras over $k$ while $C_{x}$ is
a coalgebra in the isomorphism class of $x \in I$. Now since both
forgetful functors $F_{1}$ and $F$ are faithful it follows that
$\overline{F}_{1}(G)$ and $\overline{F}(G)$ are generators in the
categories of Poisson bialgebras and respectively Poisson Hopf
algebras.
\end{proof}

\section*{Some comments and open problems}\selabel{4}

\quad The main results of this paper are \thref{1.3} and
\thref{free3} which prove the cocompleteness of the categories
$k$-BiAlgPoiss and $k$-PoissHopf, and respectively the existence
of a free Poisson Hopf algebra on every coalgebra. In order to
complete the categorical picture drawn by the above results, it is
natural to ask the following:

\textbf{Question $1$:} Are the categories $k$-BiAlgPoiss and
respectively $k$-PoissHopf complete (i.e. do they have arbitrary
products and equalizers)?

\quad It is worth pointing out that, if limits in $k$-BiAlgPoiss
and respectively $k$-PoissHopf do exist, then \thref{free3} tells
us that they should be constructed as simply the limits of the
underlying coalgebras.

\textbf{Question $2$:} Does there exist a cofree Poisson Hopf
algebra on every Poisson algebra, respectively on every Poisson
bialgebra? (i.e. do the forgetful functors $U : k$-PoissHopf $\to
k$-Poiss and respectively $\overline{U} : k$-PoissHopf $\to
k$-BiAlgPoiss have right adjoints?)

\quad Another important issue, also related to the existence of
cofree objects in the above mentioned categories, is the
injectivity (resp. surjectivity) of monomorphisms (resp.
epimorphisms). It was proven in \cite{AC1} that monomorphisms
(resp. epimorphisms) in the category $k$-HopfAlg of Hopf algebras
are not necessarily injective (reps. surjective) maps. This was
achieved by noticing that the antipode of any Hopf algebra is both
a monomorphism and an epimorphism in the category $k$-HopfAlg of
Hopf algebras together with the well known fact that there exist
Hopf algebras with non-injective, respectively non-surjective
antipode. Although the result in \cite{AC1} holds true in the
category $k$-PoissHopf of Poisson Hopf algebras as well, it has no
implications on the injectivity (resp. surjectivity) of
monomorphisms (resp. epimorphisms) due to the fact that all
Poisson Hopf algebras have bijective antipodes as a consequence of
being commutative Hopf algebras.

\quad As noted above, we do not know whether epimorphisms in the
aforementioned categories are surjective maps. This information is
of particular interest mainly in connection to the Special Adjoint
Functor Theorem (see \cite[Corollary p. 130]{ML}) which, under
some additional assumptions, provides necessary and sufficient
conditions for a functor to be a left (respectively right)
adjoint. Having a positive answer to the question of whether
epimorphisms in $k$-PoissHopf are surjective maps would also
settle the co-wellpoweredness problem into positive and therefore
all the assumptions required for applying the Special Adjoint
Functor Theorem would be fulfilled. Thus, the forgetful functors
$U$ and $\overline{U}$ would have right adjoints and therefore we
would be able to construct the cofree Poisson Hopf algebra
generated by a Poisson algebra, respectively a Poisson bialgebra.
On the other hand, a negative answer to the question of whether
epimorphisms are surjective maps in $k$-PoissHopf it has no
implications on the co-wellpoweredness of the category
$k$-PoissHopf: for instance, the category $k$-HopfAlg of Hopf
algebras is co-wellpowered although epimorphisms are not
necessarily surjective maps (see \cite{A2, AC1, HP}).

\end{document}